\documentclass[12pt]{amsart}
\usepackage{amssymb,amsmath,amsthm,mathrsfs,multirow,xcolor,framed,url}
\usepackage[pdftex,
         pdfauthor={Dandan Chen and Rong Chen},
         pdftitle={Congruence modulo 4 for Andrews' integer partition},
         pdfsubject={MATHEMATICS},
         pdfkeywords={Ternary quadratic form, Ramanujan-type congruences, Class number},
         pdfproducer={Latex with hyperref},
         pdfcreator={pdflatex}]{hyperref}
\newtheorem{theorem}{Theorem}[section]
\newtheorem{lemma}[theorem]{Lemma}
\newtheorem{cor}[theorem]{Corollary}

\newtheorem{prop}[theorem]{Proposition}

\theoremstyle{definition}
\newtheorem{definition}[theorem]{Definition}

\theoremstyle{remark}
\newtheorem{remark}[theorem]{Remark}

\numberwithin{equation}{section}

\newcommand\nutwid{\overset {\text{\lower 3pt\hbox{$\sim$}}}\nu}

\allowdisplaybreaks
\newcommand\mylabel[1]{\label{#1}}

\newcommand{\beqs}{\begin{equation*}}
\newcommand{\eeqs}{\end{equation*}}
\newcommand{\beq}{\begin{equation}}
\newcommand{\eeq}{\end{equation}}
\renewcommand{\MR}[1]{\href{http://www.ams.org/mathscinet-getitem?mr={#1}}{MR{#1}}}

\begin{document}

\title[A generalization of the Mordell integral]{A generalization of
the Mordell integral}

\author{Dandan Chen}
\address{Department of Mathematics, Shanghai University, Shanghai, People's Republic of China}
\address{Newtouch Center for Mathematics of Shanghai University, Shanghai, China}
\email{mathcdd@shu.edu.cn}
\author{Rong Chen}
\address{Department of Mathematics, Shanghai Normal University, Shanghai, People's Republic of China}
\email{rongchen20@tongji.edu.cn}
\author{Sander  Zwegers }
\address{Department of Mathematics and Computer Science, University of Cologne, Weyertal 86–90, 50931, Cologne, Germany}
\email{szwegers@uni-koeln.de}



\begin{abstract}
We find a generalization of the Mordell integral and we also establish a  set of properties for a generalization of the Mordell integral similar to those in the third author's PhD thesis.
\end{abstract}

\maketitle
\section{Introduction}
Mordell initiated the study of the integral $\int_{-\infty}^{\infty}\frac{e^{ax^2+bx}}{e^{cx}+1}dx$, $\Re(x)<0$, in his two influential papers \cite{Mordell-QJM-20, Mordell-Acta-33}. Prior to his work, special cases of this integral had been studied, for example, by Riemann in his Nachlass \cite{Siegel-32} to derive the approximate functional equation for the Riemann zeta function, by Kronecker \cite{Kronecker-JRAM-89-1, Kronecker-JRAM-89-2}  to derive the reciprocity for Gauss sums. Mordell \cite{Mordell-Acta-33} showed that the above integral can be reduced to two standard forms, namely,
\begin{align}
&\varphi(z;\tau):=\tau \int_{-\infty}^{\infty}\frac{e^{\pi i\tau x^2-2\pi zx}}{e^{2\pi x}-1}dx,\label{defn-varphi}\\
&\sigma(z;\tau):=\tau \int_{-\infty}^{\infty}\frac{e^{\pi i\tau x^2-2\pi zx}}{e^{2\pi\tau x}-1}dx\nonumber,
\end{align}
 where the path of integration may be taken as either the real axis of $\tau$ indented by the lower half a small circle described about the origin as centre, and was the first to study the properties of these integrals with respect to modular transformations. Bellmann \cite[p. 52]{Bellman} coined the terminology 'Mordell integrals' for these types of integrals.

Mordell integrals play a very important role in the groundbreaking Ph.D. thesis of the third author \cite{Zwegers-thesis} which sheds a clear light on Ramanujan's mock theta functions. The definition of a Mordell integral $h(z;\tau)$, $\Im(\tau)>0$, employed by the third author \cite[p. 6]{Zwegers-thesis}, and now standard in the contemporary literature is,
\begin{align*}
h(z;\tau):=\int_{-\infty}^{\infty}\frac{e^{\pi i\tau x^2-2\pi zx}}{\cosh {\pi x}}dx.
\end{align*}
As remarked in \cite[p. 5]{Zwegers-thesis}, $h(z;\tau)$ is essentially the function $\varphi(z;\tau)$ defined in \eqref{defn-varphi}, i.e.,
\begin{align*}
h(z;\tau)=\frac{-2i}{\tau}e^{-(\pi i\tau/4+\pi iz)}\varphi(z+\frac{\tau-1}{2};\tau).
\end{align*}
The properties of the function $h(z;\tau)$ can be found in \cite[p. 6]{Zwegers-thesis}, i.e.

\begin{prop}The function $h$ has the following properties:
\begin{align*}
&\text{(1)}\quad h(z)+h(z+1)=\frac{2}{\sqrt{-i\tau}}e^{\pi i(z+1/2)^2/\tau},\\
&\text{(2)}\quad h(z)+e^{-2\pi iz-\pi i\tau}h(z+\tau)=2e^{-\pi iz-\pi i\tau/4},\\
&\text{(3)}\quad z\rightarrow h(z;\tau) \quad \text{is the unique holomorphic function satisfying (1) and (2)},\\
&\text{(4)}\quad h \quad \text{is an even function of z},\\
&\text{(5)}\quad h\left(\frac z\tau;-\frac1\tau\right)=\sqrt{-i\tau}e^{-\pi iz^2/\tau}h(z;\tau),\\
&\text{(6)}\quad h(z;\tau)=e^{\pi i/4}h(z;\tau+1)+e^{-\pi i/4}\frac{e^{\pi iz^2/{(\tau+1)}}}{\sqrt{\tau+1}}h\left(\frac{z}{\tau+1};\frac{\tau}{\tau+1}\right).
\end{align*}
\end{prop}
We refer the reader to a more recent article and references therein for further applications of the Mordell integrals \cite{Choi,Dixit-17}.

Let
\begin{align*}
h_N(z;\tau, \chi):=\int_{-\infty}^{\infty}\Phi_N(x;\chi)e^{\pi i\tau x^2-2\pi zx}dx,
\end{align*}
where $\Phi_N(x;\chi)$ is defined in \eqref{defn-Phi}, and $\tau\in \mathbb{H}, z\in\mathbb{C}$. For example, when $N=3$, we have
\begin{align*}
h_3(z;\tau,\chi):=\int_{-\infty}^{\infty}\frac{e^{\pi i\tau x^2-2\pi zx}}{1+2\cosh {\frac{2\pi x}{\sqrt 3}}}dx
\end{align*}
and for $N=4$, we have
\begin{align*}
h_4(z;\tau,\chi)=\frac 12 h(z;\tau).
\end{align*}
From this perspective, we consider $h_N(z;\tau,\chi)$ a natural generalization of the Mordell integral. The main goal of this paper is to establish a similar set of properties for a generalization of the  Mordell integral.

We first note
\begin{theorem}\label{Thm-hN-EO}
 $h_N(z;\tau,\chi)$ is an entire function of $z$; moreover, if $\chi(-1)=-1$, $h_N(z;\tau,\chi)$
is an even function of $z$ and if $\chi(-1)=1$, $h_N(z;\tau,\chi)$ is an odd function of $z$.
\end{theorem}
We now investigate the transformations of $h_N(z;\tau,\chi): z\rightarrow z+\sqrt N$ and $z\rightarrow z+\sqrt N\tau$.

\begin{theorem}\label{Thm-hN-1-2}we have
\begin{align}
&h_N(z;\tau,\chi)-h_N(z+\sqrt N;\tau,\chi)
=\frac{\chi(-1)}{\sqrt{-i\tau}}\sum_{k=1}^{N-1}\chi(k)e^{i\pi(z+\frac{k}{\sqrt N})^2/\tau};\mylabel{h-z+N}\\
&h_N(z;\tau,\chi)-e^{-i\pi N\tau-2i\pi\sqrt N z}h_N(z+\sqrt N\tau;\tau,\chi)
=\frac{g(\chi)}{i\sqrt N}\sum_{n=1}^{N-1}\overline{\chi(n)}e^{-i\pi n^2\tau/N-2i\pi nz/\sqrt N}.\mylabel{h-z+Ntau}
\end{align}
where $g(\chi)$ is defined in \eqref{defn-gchi}.
\end{theorem}

Next,  we have the transformation of $h_N(z;\tau,\chi):\tau\rightarrow -\frac1\tau$.
\begin{theorem}\label{Thm-hN-1tau}We have
\begin{align*}
h_N\left(\frac{z}{\tau};-\frac{1}{\tau},\chi\right)
=\frac{i\sqrt N}{g(\overline\chi)}e^{-i\pi z^2/\tau}\sqrt{-i\tau}h_N(-z;\tau,\overline\chi).
\end{align*}
When $\chi$ is a primitive character modulo $N$, we have
\begin{align*}
h_N\left(\frac{z}{\tau};\frac{-1}{\tau},\chi_N\right)
=e^{-\pi iz^2/\tau}\sqrt{-i\tau}h_N(z;\tau,\chi_N).
\end{align*}
\end{theorem}

The paper is organized as follows. In Section \ref{sec-pre}, we first  recall some basic properties characters modulo $N$ and the $\chi$-twisted geometric series. We also discuss the  fourier transform of $\Phi_N(x;\chi)$.  In Section \ref{Thm-proof} we  prove  Theorem \ref{Thm-hN-EO}, Theorem \ref{Thm-hN-1-2} and Theorem \ref{Thm-hN-1tau}.
\section{Preliminaries}\label{sec-pre}
\subsection{Properties of characters modulo $N$ }
Let $N$ be a positive integer. We say $\chi$ be a character modulo $N$, if for all integers $m$ and $n$, we have
\begin{itemize}
\item $\chi(1)=1$,
\item $\chi(n+N)=\chi(n)$,
\item $\chi(mn)=\chi(m)\chi(n)$,
\item $\chi(n)=0$ if and only if $\gcd(n,N)>1$.
\end{itemize}

Let $N^\prime$ be a positive integer which is divisible by $N$. For any character $\chi$ modulo $N$, we can form a character $\chi^\prime$ modulo $N^\prime$ as follows:
\begin{align*}
\chi^\prime(n)=\begin{cases} \chi(n),&\text{if $\gcd(n,N^{\prime})=1$};\\
0, &\text{if $\gcd(n,N^{\prime})>1$}.
\end{cases}
\end{align*}
We say that $\chi^\prime$ is induced by the character $\chi$. Let $\chi$ be a character modulo $N$. If there is a proper division $d$ of $N$ and a character $\chi_1$ modulo $d$ which induces $\chi$, then the character $\chi$ is said to be non-primitive, otherwise it is called primitive. A character $\chi_0$ is called the unit character modulo $N$ if $\chi_0(n)=1$ for $\gcd(n,N)=1$. Let $\chi_d(n)=\left(\frac {d}{n}\right)$ denote the Kronecker's extension of Jacobi symbol \cite[p. 35]{Borevich}. It is well known that $\chi_d$ is primitive modulo $|d|$ \cite[Theorem 5]{Cohn}. For $d=1$, we define $\chi_1(n)=1$ for all $n$.

The Gaussian sum $g_n(\chi)$ associated with the character $\chi$ is defined as
\begin{align*}
g_n(\chi)=\sum_{k=1}^{N-1}\chi(k)\exp(2ikn\pi/N).
\end{align*}
We denote
\begin{align}\label{defn-gchi}
g(\chi)=g_1(\chi).
\end{align}
We need a lemma \cite[p.334]{Bump-97}.
\begin{lemma} Let $\chi$ be a primitive character modulo $N$. Then for any integer $n$, we have
\begin{align*}
g_n(\chi)=\overline{\chi(n)}g(\chi).
\end{align*}
\end{lemma}
We remark for $\gcd(N,n)=1$, the above equality holds without $\chi$ being primitive; the primitiveness of $\chi$ is needed in showing that the left hand side of the identity is zero when $\gcd(N,n)>1$.

\subsection{$\chi$-twisted geometric series}
Consider
\begin{align*}
\phi_N(t;\chi)=\frac{\sum_{n=1}^{N-1}\chi(n)t^n}{1-t^N}.
\end{align*}
If there is no danger of confusion, we will simply write $\phi_N(t)$ for $\phi_N(t;\chi)$. The series $\sum_{n}^{}\chi(n)a_n$
is said to be the $\chi$-twisted series of $\sum_{n}^{}a_n$. The function $\phi_N$ is generated by twisting the geometric series $\sum_{n=0}^{\infty}t^n$ with the character $\chi$:
\begin{align*}
\sum_{n=0}^{\infty}\chi(n)t^n=\sum_{k=0}^{\infty}\sum_{n=0}^{N-1}\chi(kN+n)t^{kN+n}
=\sum_{n=0}^{N-1}\chi(n)t^n\sum_{k=0}^{\infty}t^{kN}=\phi_{N}(t;\chi),
\end{align*}
where $|t|<1$.

It is easy to verity
\begin{lemma}\mylabel{lem-phiN}\cite[Lemma 0.2]{Shen-ecnu}  Suppose $\chi\neq\chi_0$. We have

(1)
If $\chi(-1)=-1$,
\begin{align*}
\phi_N(t)=\phi_N(t^{-1}).
\end{align*}

(2)
If $\chi(-1)=1$,
\begin{align*}
\phi_N(t)=-\phi_N(t^{-1}).
\end{align*}

(3)
If $N\neq 1,2$,
\begin{align*}
\phi_N(1)=-\frac{1}{N}\sum_{n=1}^{N-1}n\chi(n).
\end{align*}

Moreover, $\phi_N(1)=0$ \text{if $\chi(-1)=1$}.
\end{lemma}

\subsection{Fourier transform of $\Phi_N(x;\chi)$}
The discrete Fourier transform (DFT) of a sequence $\mathbf{x}=(x_0,x_1,\cdots,x_{N-1})\in\mathbb{C}^N$ of length $N$ is $X=(X_0,X_1,\cdots,X_{N-1})$, where
\begin{align*}
X_k=\sum_{n=0}^{N-1}x_ne^{-2\pi ikn/N}.
\end{align*}
In this note we consider $\mathbf{x}=(x_0,x_1,\cdots,x_{N-1})$ for which $x_0=\sum_{n=1}^{N-1}x_n=0$ and denote the space of such sequences by $\mathbb{X}^N$. The DFT maps $\mathbb{X}^N$ into itself: we have
\begin{align*}
X_0=\sum_{n=0}^{N-1}x_n=0
\end{align*}
and
\begin{align*}
\sum_{n=1}^{N-1}X_n=\sum_{n=0}^{N-1}X_n=\sum_{n=0}^{N-1}\sum_{k=0}^{N-1}x_ke^{-2\pi ink/N}=\sum_{k=0}^{N-1}x_k\sum_{n=0}^{N-1}e^{-2\pi ikn/N}=Nx_0=0.
\end{align*}
For convenience we normalize the DFT by $\sqrt N$ and consider $\mathcal{F}_d:\mathbb{X}^N\rightarrow\mathbb{X}^N$ given by $\mathcal{F}_d(\mathbf{x})=\mathbf{X}/\sqrt N$. Further, we consider the map $\mathcal{S}_d:\mathbb{X}^N\rightarrow\mathbb{X}^N$ given by $x_k\rightarrow x_{N-k}$. These maps satisfy $\mathcal{F}^2_d=\mathcal{S}_d$. Hence $\mathcal{F}_d$ is a bijection with inverse $\mathcal{F}_d^{-1}=\mathcal{F}_d\circ\mathcal{S}_d$.

Let $\mathcal{S}$ denote the space of Schwartz functions. For $f\in\mathcal{S}$ the Fourier transform $\mathcal{F}(f)\in\mathcal{S}$ is given by
\begin{align*}
\mathcal{F}(f)(x)=\int_{\mathbb{R}}f(t)e^{-2\pi ixt}dt.
\end{align*}
The map $\mathcal{F}:\mathcal{S}\rightarrow\mathcal{S}$ satisfies $\mathcal{F}^2={S}$, where ${S}:\mathcal{S}\rightarrow\mathcal{S}$, $S(f)(x)=f(-x)$.

The convolution of $f_1$ and $f_2$ is by definition
\begin{align*}
(f_1*f_2)(x)=\int_{-\infty}^{\infty}f_1(x-t)f_2(t)dt.
\end{align*}
and $\mathcal{F} (f*g)=\mathcal{F} f\cdot\mathcal{F} g$.
\begin{definition}
We define $\Phi:\mathbb{X}^N\rightarrow\mathcal{S}$ by
\begin{align*}
\Phi(\mathbf{x})(t)=\frac{\sum_{k=1}^{N-1}x_ke^{2\pi kt/\sqrt N}}{1-e^{2\pi t\sqrt N}}
\end{align*}
\end{definition}

\begin{remark}
(a) It's easy to check that $\Phi(\mathbf{x})$ is indeed a Schwartz functin. The denominator has a simple zero for $t=0$, but there is no singularity since the numerator also has a zero. Here the condition $\sum_{n=1}^{N-1}x_n=0$ comes in. Further, $\Phi(\mathbf{x})$ has exponential decay for $t\rightarrow+\infty$ since $k<N$ and for $t\rightarrow-\infty$ since $x_0=0$.

(b)For all $\mathbf{x}\in\mathbb{X}^N$ we have
\begin{align*}
\mathcal{S}(\Phi(\mathbf{x}))=-\Phi(\mathcal{S}_d(\mathbf{x})).
\end{align*}
\end{remark}

\begin{theorem}\label{thm-fphi}
For all $\mathbf{x}\in\mathbb{X}^N$ we have
\begin{align*}
\mathcal{F}(\Phi(\mathbf{x}))=-i\Phi(\mathcal{F}_d(\mathcal{S}_d(\mathbf{x})))
=-i\Phi(\mathcal{F}^{-1}_d(\mathbf{x})).
\end{align*}
\end{theorem}
\begin{proof}
By definition we have
\begin{align*}
\mathcal{F}(\Phi(\mathbf{x}))(x)
=\int_\mathbb{R}\Phi(\mathbf{x})(t)e^{-2\pi ixt}dt.
\end{align*}
Now consider the integral
\begin{align*}
\int_{i\sqrt N+\mathbb{R}}\Phi(\mathbf{x})(t)e^{-2\pi ixt}dt.
\end{align*}
Shifting the integration variable $t$ by $i\sqrt N$ we see that it equals
\begin{align*}
e^{2\pi x\sqrt N}\int_\mathbb{R}\Phi(\mathbf{x})(t+i\sqrt N)e^{-2\pi ixt}dt.
\end{align*}
Since both $e^{2\pi kt/\sqrt N}$ and $e^{2\pi t\sqrt N}$ are $i\sqrt N$ periodic in $t$, we have $\Phi(\mathbf{x})(t+i\sqrt N)=\Phi(\mathbf{x})(t)$ and so we obtain
\begin{align*}
\int_{i\sqrt N+\mathbb{R}}\Phi(\mathbf{x})(t)e^{-2\pi ixt}dt
=e^{2\pi x\sqrt N}\int_{\mathbb{R}}\Phi(\mathbf{x})(t)e^{-2\pi ixt}dt.
\end{align*}
With the residue theorem we then get
\begin{align*}
(1-e^{2\pi x\sqrt N})\mathcal{F}(\Phi(\mathbf{x}))(x)
=&\left(\int_{\mathbb{R}}-\int_{i\sqrt N+\mathbb{R}}\right)\Phi(\mathbf{x})(t)e^{-2\pi ixt}\\
=&2\pi i\sum_{n=1}^{N-1}\mathop{Res}\limits_{t=\frac{in}{\sqrt N}}\Phi(\mathbf{x})(t)e^{-2\pi ixt}.
\end{align*}
These residues we can easily compute:
\begin{align*}
\mathop{Res}\limits_{t=\frac{in}{\sqrt N}}\Phi(\mathbf{x})(t)e^{-2\pi ixt}
=&\mathop{Res}\limits_{t=\frac{in}{\sqrt N}}\frac{\sum_{k=1}^{N-1}x_ke^{2\pi kt/\sqrt N}}{1-e^{2\pi t\sqrt N}}e^{-2\pi ixt}\\
=&\sum_{k=1}^{N-1}x_ke^{2\pi nx/\sqrt N}\mathop{Res}\limits_{t=\frac{in}{\sqrt N}}\frac{1}{1-e^{2\pi t\sqrt N}}\\
=&-\frac{1}{2\pi\sqrt N}e^{2\pi nx/\sqrt N}\sum_{k=1}^{N-1}x_ke^{2\pi i\frac{kn}{N}}\\
=&-\frac{1}{2\pi\sqrt N}e^{2\pi nx/\sqrt N}\sum_{k=1}^{N-1}x_{N-k}e^{-2\pi i\frac{kn}{N}}.
\end{align*}
Hence we have
\begin{align*}
\mathcal{F}(\Phi(\mathbf{x})(x))=\frac{-i}{\sqrt N}\frac{\sum_{n=1}^{N-1}\sum_{k=1}^{N-1}x_{N-k}e^{-2\pi i\frac{kn}{N}}e^{2\pi nx/\sqrt N}}{1-e^{2\pi x\sqrt N}},
\end{align*}
which gives the desired result.
\end{proof}

We define
\begin{align}\label{defn-Phi}
\Phi_N(x;\chi)=\phi_N(e^{2\pi x/\sqrt N};\chi).
\end{align}
Again, if there is no confusion, we abbreviate $\Phi_N(x;\chi)$ by $\Phi_N(x)$. We have
\begin{align*}
&\Phi_3(x)=\frac{1}{1+2\cosh \frac{2\pi x}{\sqrt 3}},\\
&\Phi_4(x)=\frac{1}{2\cosh \pi x}.
\end{align*}
where $\cosh x=\frac{e^x+e^{-x}}{2}$.
From \cite[Theorem 3.1]{Shen-ecnu}, we have
\begin{theorem}Let $N\geq3$ be the conductor of $\chi$. Then
\begin{align}\label{eq-phi-trans}
\mathcal{F}({\Phi}_N(\cdot;\chi))=\frac{g(\chi)}{i\sqrt N}\Phi_N(\overline \chi)
\end{align}
\end{theorem}
\begin{proof}
Let $\chi$ be a non-principal character modulo $N$ and consider $\mathbf{x}=\chi(k)$. Since $\sum_{n=1}^{N-1}\chi(n)=0$ we have $\mathbf{x}\in\mathbb{X}^{N}$. For the discrete Fourier transform of $\mathcal{S}_d(\mathbf{x})$ we have
\begin{align*}
X_k=\sum_{n=0}^{N-1}\chi(N-n)e^{-2\pi i\frac{kn}{N}}=\sum_{n=0}^{N-1}\chi(n)e^{2\pi i\frac{kn}{N}}
\end{align*}
If $\chi$ is a primitive character this sum equals
\begin{align*}
\overline\chi(k)\sum_{n=1}^{N-1}\chi(n)e^{2\pi i\frac{kn}{N}}
=\overline\chi(k)g(\chi).
\end{align*}
Hence
\begin{align*}
\mathcal{F}_d(\mathcal{S}_d(\chi))=\frac{g(\chi)}{\sqrt N}\overline\chi
\end{align*}
and using Theorem \ref{thm-fphi}, \eqref{eq-phi-trans} holds.
\end{proof}

Let $-D$ be the discriminant of the imaginary quadratic field $\mathbb(\sqrt{-D})$ and let $\chi_D(n)=\left(\frac{-D}{n}\right)$. Then $\chi_D$ is a non-unit primitive character modulo $D$ and $\chi(-1)=-1$; moreover, since $g(\chi_D)=i\sqrt D$, we obtain a family of functions which are invariant under Fourier transform.
\begin{cor} We have
\begin{align*}
\mathcal{F}({\Phi}_D(\chi_D)){(t)}=\Phi_D(t;\chi_D).
\end{align*}
\end{cor}

\begin{lemma}\mylabel{lem-F-residue}
Let
\begin{align*}
F(x;\tau):=e^{i\pi\tau x^2-2\pi zx}\Phi_N(x;\tau),
\end{align*}
where $x\in\mathbb{C}$. The residue of $F$ at $x_n=\frac{in}{\sqrt N}, n=1,2,\cdots,N-1$, is
\begin{align*}
Res(F,x_n)=-\frac{g(\chi)\overline{\chi(n)}}{2\pi\sqrt N}e^{-i\pi n^2\tau/N-2\pi inz/\sqrt N}.
\end{align*}
\end{lemma}

\begin{proof}Since
\begin{align*}
Res(F,x_n)=\lim_{x\rightarrow x_n}(x-x_n)F(x)
=e^{-i\pi\tau n^2/N-2\pi inz/\sqrt N}\lim_{x\rightarrow x_n}(x-x_n)\Phi_N(x;\chi)
\end{align*}
and
\begin{align*}
&(x-x_n)\Phi_N(x;\chi)
=\frac{x-\frac{in}{\sqrt N}}{1-e^{2\pi\sqrt N x}}\sum_{k=1}^{N-1}\chi(k)e^{2\pi kx/\sqrt N},\\
&\lim_{x\rightarrow x_n}(x-x_n)\Phi_N(x;\chi)
=-\frac{1}{2\pi\sqrt N}\sum_{k=1}^{N-1}\chi(k)e^{2\pi ink/N}
=-\frac{g(\chi)\overline{\chi(n)}}{2\pi\sqrt N},
\end{align*}
we have
\begin{align*}
Res(F,x_n)=-\frac{g(\chi)\overline{\chi(n)}}{2\pi\sqrt N}e^{-\pi i\tau n^2/N-2\pi inz/\sqrt N}.
\end{align*}
\end{proof}
\section{The proof of Theorem \ref{Thm-hN-EO}--\ref{Thm-hN-1tau}}\label{Thm-proof}
\begin{proof} [The proof of Theorem \ref{Thm-hN-EO}]
From Lemma \ref{lem-phiN},  when $\chi(-1)=1$, we know that $\phi_N(t)=-\phi_N(t^{-1})$. Then
\begin{align*}
h_N(-z;\tau, \chi)&=\int_{-\infty}^{\infty}\Phi_N(x;\chi)e^{\pi i\tau x^2+2\pi zx}dx\\
&=-\int_{-\infty}^{\infty}\Phi_N(-x;\chi)e^{\pi i\tau x^2-2\pi zx}dx\\
&=-h_N(z;\tau, \chi).
\end{align*}
The above last equation holds by $\Phi_N(-x;\chi)=\phi(e^{-2\pi x/{\sqrt N}})=\phi(e^{2\pi x/{\sqrt N}})=\Phi_N(x;\chi)$. We omit the proof of the case of $\chi(-1)=-1$.
\end{proof}

\begin{proof}[The proof of Theorem \ref{Thm-hN-1-2}]

(1)
Using
\begin{align*}
(1-e^{2\sqrt N\pi x})\Phi_N(x;\chi)
=\sum_{k=1}^{N-1}\chi(k)e^{2k\pi x/\sqrt N},
\end{align*}
 we have
\begin{align}
h_N(z;\tau,\chi)-h_N(z-\sqrt N;\tau,\chi)
=&\int_{-\infty}^{\infty}e^{i\pi\tau x^2-2\pi zx}\Phi_N(x;\chi)(1-e^{2\sqrt N\pi x})dx\nonumber\\
=&\int_{-\infty}^{\infty}e^{i\pi\tau x^2-2\pi zx}\sum_{k=1}^{N-1}\chi(k)e^{2k\pi x/\sqrt N}dx\nonumber\\
=&\frac{1}{\sqrt{-i\tau}}\sum_{k=1}^{N-1}\chi(k)e^{i\pi(z-\frac{k}{\sqrt N})^2/\tau}\label{hN-1}.
\end{align}
The last equation comes from the well-known identity
\begin{align*}
\int_{-\infty}^{\infty}e^{\pi i\tau x^2-2\pi zx}dx=\frac{e^{\pi iz^2/\tau}}{\sqrt{-i\tau}}.
\end{align*}
Replacing $z$ by $z+\sqrt N$  in \eqref {hN-1} and simplify, \eqref{h-z+N} holds.

(2) If we change $x$ into $x+i\sqrt N$ we find
\begin{align*}
&\int_{-\infty+i\sqrt N}^{\infty+i\sqrt N}e^{\pi i\tau x^2-2\pi zx}\Phi_N(x;\chi)dx\\
=&\int_{-\infty}^{\infty}e^{\pi i\tau(x+i\sqrt N)^2-2\pi z(x+i\sqrt N)}\Phi_N(x+i\sqrt N;\chi)dx\\
=&e^{-\pi iN\tau-2\pi i\sqrt N z}\int_{-\infty}^{\infty}e^{\pi i\tau x^2-2\pi x(z+\sqrt N\tau)}\Phi_N(x;\chi)dx\\
=&e^{-\pi iN\tau-2\pi i\sqrt N z}h_N(z+\sqrt N\tau;\tau,\chi),
\end{align*}
 by $\Phi(z+i\sqrt N;\chi)=\Phi(z;\chi)$. From Lemma \ref{lem-F-residue}, using Cauchy's theorem we find
\begin{align*}
&h_N(z;\tau,\chi)-e^{-i\pi N\tau-2i\pi\sqrt N z}h_N(z+\sqrt N\tau;\tau,\chi)\\
=&\left(\int_{-\infty}^{\infty}-\int_{-\infty+i\sqrt N}^{\infty+i\sqrt N}\right)e^{\pi i\tau x^2-2\pi zx}\Phi_N(x;\chi)dx\\
=&2\pi i\sum_{n=1}^{N-1}Res(F,x_n)\\
=&\frac{g(\chi)}{i\sqrt N}\sum_{n=1}^{N-1}\overline{\chi(n)}e^{-\pi in^2\tau/N-2\pi inz/\sqrt N}.
\end{align*}
\end{proof}

\begin{proof}[The proof of Theorem \ref{Thm-hN-1tau}]
Let $f_\tau(x)=e^{\pi i\tau x^2}$, $\tau\in\mathbb{H}$. The Fourier transform of $f_\tau$ is given by
\begin{align*}
\mathcal{F}{f}_\tau=\frac{1}{\sqrt{-i\tau}}f_{-\frac 1\tau}.
\end{align*}
Next, we have
\begin{align*}
h_N(iz;\tau,\chi)&=\int_{-\infty}^{\infty}e^{-2\pi i zx}\Phi_N(x;\chi)e^{\pi i\tau x^2}dx=\mathcal{F}(f_\tau\Phi_N)(z)\\
&=(\mathcal{F} f_\tau*\mathcal{F}\Phi_N)(z)=\frac{g(\chi)}{i\sqrt N}\frac{1}{\sqrt{-i\tau}}
\int_{-\infty}^{\infty}\Phi_N(x;\overline\chi)e^{-\pi i(z-x)^2/\tau}dx\\
&=\frac{g(\chi)}{i\sqrt N}\frac{e^{-\pi iz^2/\tau}}{\sqrt{-i\tau}}
\int_{-\infty}^{\infty}\Phi_N(x;\overline\chi)e^{\pi i(-1/\tau)x^2-2\pi(-iz/\tau)x}dx\\
&=\frac{g(\chi)}{i\sqrt N}\frac{e^{-\pi iz^2/\tau}}{\sqrt{-i\tau}}h_N\left(-\frac{iz}{\tau};-\frac{1}{\tau},\overline\chi\right).
\end{align*}
Thus, for $z\in\mathbb{R}$,
\begin{align*}
h_N(iz;\tau,\chi)=\frac{g(\chi)}{i\sqrt N}\frac{e^{i \pi (iz)^2/\tau}}{\sqrt{-i\tau}}h_N\left(-\frac{iz}{\tau};-\frac{1}{\tau},\overline\chi\right).
\end{align*}
Since $h_N$ is an entire function of $z$, the identity holds for all $z\in\mathbb{C}$. Replacing $iz$ by $z$,
we have the  desired identity.

\end{proof}

\subsection*{Acknowledgements}

The authors would like to thank Prof. Li-Chien Shen  for helpful discussion. The first author was supported in part by  the National Natural Science Foundation of China (\#12201387), Shanghai Sailing Program (\#21YF1413600) and  Shanghai Key Laboratory of Pure Mathematics and Mathematical Practice(\#22DZ2229014).

\end{document}